\documentclass[a4paper,10pt,reqno, english]{amsart}

\usepackage{amsmath,amssymb,amscd,amsthm,amsfonts}
\usepackage{graphicx,subfigure}
\usepackage{hyperref}
\usepackage{dsfont}
\usepackage[nobysame, alphabetic]{amsrefs}
\usepackage{tikz}
\usepackage[capitalise]{cleveref}
\usepackage{mathrsfs}
\usepackage{booktabs,array}
\usepackage{microtype}
\usepackage{enumitem}

\newtheorem{theorem}{Theorem}
\newtheorem{lemma}{Lemma}

\newtheorem{definition}{Definition}

\def\rr{\mathds{R}}
\def\hh{\mathds{H}}

\DeclareMathOperator{\sgn}{sgn}

\DeclareMathOperator{\Spin}{Spin}
\DeclareMathOperator{\conv}{conv}

\DeclareMathOperator{\Sym}{Sym}

\title{Four hyperplanes do not always equipartition a mass in $\mathbb{R}^4$}

\hypersetup{
  pdftitle={Four hyperplanes do not always equipartition a mass in $R^4$},
  pdfauthor={Pablo Soberon}
}

\author[Sober\'on]{Pablo Sober\'on}\address{Baruch College \& The Graduate Center, City University of New York, One Bernard Baruch Way, New York, NY 10010, United States} 
\email{psoberon@gc.cuny.edu}

\thanks{
The research of P. Sober\'on is supported by NSF CAREER grant DMS-2237324 and a PSC-CUNY Track 1 award.}



\begin{document}

\begin{abstract}
We construct a smooth strictly positive density in $\mathbb{R}^4$ that cannot be divided into $16$ parts of the same size by four affine hyperplanes.   This settles the last open case of Gr\"unbaum's 1960 conjecture and disproves Ramos' general conjecture on hyperplane equipartitions.

We reduce the construction to finding two homogeneous polynomials in four variables, of degrees three and four, whose multilinear coefficients cannot vanish simultaneously after any orthogonal change of coordinates. We give two proofs of this nonvanishing result.  The first uses a local perturbation argument.  The second reduces it to the absence of a common zero for five explicit polynomials on \([-1,1]^6\), verified by a computer-assisted Bernstein subdivision argument.
\end{abstract}

\maketitle

\section{Introduction}

Mass partition problems are a central topic in topological combinatorics.  The general goal is, given rules on how we are allowed to split an ambient space, usually $\rr^d$, to guarantee the existence of an equipartition of a set of masses, meaning that all resulting parts have equal measure with respect to each mass.  A \textit{mass} in $\rr^d$ is a finite Borel measure that vanishes on every affine hyperplane.  The quintessential example is the ham sandwich theorem, which states that \textit{for any $d$ masses in $\rr^d$ there exists a hyperplane simultaneously halving all of them}.  This was proved in dimension three by Banach using the Borsuk--Ulam theorem and in general by Stone and Tukey \cites{Steinhaus1938, Stone1942}.  Mass partition problems provide compelling motivation for the use of topological tools, and have deep connections to computational geometry, discrete geometry, and mathematical economics \cites{Matousek2003, RoldanPensado2022}.

A classic problem in the area is Gr\"unbaum's 1960 conjecture \cite{Grunbaum1960}*{Rmk. (v)}.  Gr\"unbaum conjectured that every mass in $\rr^d$ could be partitioned into $2^d$ parts of equal size by $d$ hyperplanes.  Hadwiger proved the conjecture for $d\le 3$ \cite{Hadwiger1966}, while Avis constructed counterexamples for $d \ge 5$ \cite{Avis1984}.  Avis' counterexamples rely on concentrating a mass around the moment curve in $\rr^d$.  The case $d=4$ resisted both approaches.  Gr\"unbaum's conjecture has had multiple extensions and related problems, and has motivated significant development in topological combinatorics and computational geometry \cites{Yao1989, Ramos1996, ManiLevitska2006, Blagojevic2016}.  The lucid survey by Blagojevi\'c, Frick, Haase, and Ziegler collects the history and progress around this problem \cite{Blagojevic2018a}.  More recently, results around Gr\"unbaum's conjecture include additional conditions \cites{Maldonado2025, Aronov2026, Manta2024}, such as orthogonality, or generalizations to mass assignments \cite{Blagojevic2023, Blagojevic2025a}.  Gr\"unbaum's conjecture has been called one of the most famous open problems in topological combinatorics \cites{Simon2019, fakhrutdinov2025algorithms}.  Our main result is the following.

\begin{theorem}\label{thm:main}
    There exists a smooth strictly positive density in $\rr^4$ that admits no equipartition by four hyperplanes into $16$ parts of equal size.
\end{theorem}

Positive results in dimension four were previously known under additional hypotheses \cite{DimitrijevicBlagojevic2009} (see also \cite{Zivaljevic2008} and the follow-up discussion in \cite{Blagojevic2018a}).

If we denote by $\Delta(j,k)$ the minimum dimension in which any $j$ masses can be equipartitioned into $2^k$ sets by $k$ hyperplanes, our result shows that $\Delta(1,4) = 5$.  \cref{thm:main} is the statement $\Delta(1,4) > 4$, and the bound $\Delta(1,4) \le \Delta(2,3) =5$ settles the value \cite{Blagojevic2016}.  Ramos conjectured that $\Delta(j,k) = \left\lceil \frac{2^k-1}{k}\cdot j\right\rceil$ \cite{Ramos1996}.  \cref{thm:main} also disproves Ramos' conjecture.

\subsection{Main ideas}

The main goal is to construct the density function.  The construction of the density is based on a small perturbation of the density of a centered normal distribution.  Let $\gamma$ denote the standard centered Gaussian probability measure on $\rr^4$.  For $\varepsilon >0$, we define $\mu=\mu_{\varepsilon}$ by 
\[
d{\mu} (x) = \Big(1+\varepsilon \cdot \eta(\|x\|^2) \big(P_A(x)+ P_B(x)\big)\Big)d{\gamma}(x),
\]

where $P_A$ is a homogeneous polynomial of degree $3$  and $P_B$ is a homogeneous polynomial of degree $4$.  The term $\eta$ is a radial cutoff, so the perturbation is only done close to the origin.  The function $\eta$ is simply an element of $C^{\infty}([0,\infty))$ such that $\eta \ge 0$ and for some $R>0$ we have
\[
\eta(r) = \begin{cases}
    1 & \mbox{ if }r < R \\
    0 & \mbox{ if }r> R+1
\end{cases}
\]
This is just so that we can guarantee that, provided $\varepsilon$ is small enough, the density $d\mu$ is always positive.  The two polynomials $P_A$ and $P_B$ are the key pieces of the construction.

If $H_1,\dots, H_4$ are hyperplanes in $\rr^4$ that equipartition $\gamma$, then they must all go through the origin and their four unit vectors $q_1,\dots, q_4$ are orthonormal.  In other words, $(q_1,\dots, q_4) \in O(4)$.  Therefore, if $\mu$ is sufficiently similar to $\gamma$, then any four hyperplanes that equipartition it must be close enough to the origin and their normal vectors must be almost orthogonal.  This will allow us to parametrize the candidates for equipartitions by $O(4)$.  To find the correct perturbation, we will prove and use the following theorem.  We denote the space of symmetric $n$-linear forms in $\rr^4$ by $\Sym^n((\rr^4)^*)$.  We call the elements in $\Sym^n((\rr^4)^*)$ tensors.  There is a unique correspondence between homogeneous polynomials of degree $n$ and tensors in $\Sym^n((\rr^4)^*)$.  In particular, for $A \in \Sym^3((\rr^4)^*)$ and $B \in \Sym^4((\rr^4)^*)$ we define their polynomials $P_A(x) = A(x,x,x)$ and $P_B(x) = B(x,x,x,x)$.

\begin{theorem}\label{thm:vanishing-tool}
    There exist a cubic tensor $A \in \Sym^3 ((\rr^4)^*)$ and a quartic tensor $B \in \Sym^4 ((\rr^4)^*)$ such that the map
    \[
    (q_0,q_1,q_2,q_3) \longmapsto \begin{bmatrix}
       A(q_1,q_2,q_3) \\
       A(q_0,q_2,q_3) \\
       A(q_0,q_1,q_3) \\
       A(q_0,q_1,q_2) \\
       B(q_0,q_1,q_2,q_3)
    \end{bmatrix}
    \]
    does not have a zero on $O(4)$.
\end{theorem}

We will use $A$ and $B$ to determine $P_A$ and $P_B$, and finish the construction.  We present two proofs of \cref{thm:vanishing-tool}.  The first proof can be checked by purely analytic methods.  The second proof relies on a computer-assisted verification that a different pair of tensors $A,B$ indeed induce a function without zeros.  We present both proofs since the reduction to a finite number of cases of the second proof is interesting by itself.  The analytic proof is contained in \cref{sec:non-computational} and the computer-assisted proof is contained in \cref{sec:computational} and \cref{sec:final-computer-part}.  We keep the computer-assisted proof in the manuscript as the reduction is interesting on its own and it was the first proof we found for this result.

A second interpretation of \cref{thm:vanishing-tool} is the following.  We want to find a homogeneous polynomial $P_A$ of degree $3$ in four variables and a homogeneous polynomial $P_B$ of degree $4$ in four variables such that there exists no orthonormal change of basis that makes the four multilinear terms of $P_A$ vanish and the single multilinear term of $P_B$ vanish.

\section{Reduction to \cref{thm:vanishing-tool}}\label{sec:reduction}

We can parametrize oriented hyperplanes in $\rr^4$ by pairs $(u,t) \in S^3 \times \rr$, where $S^3$ is the unit sphere in $\rr^4$.  Each oriented hyperplane $H(u,t)$ also induces a sign function $S(u,t)$.  More precisely, we have the oriented hyperplane $H=H(u,t)$ and its sign function written as 
\[
H(u,t) = \{x \in \rr^4 : \langle x, u \rangle = t\} \qquad S_{u,t}(x) = \sgn (\langle x, u \rangle -t)
\]
To simplify the notation, given a four-tuple $\mathcal{H} = \Big((u_0,t_0), (u_1,t_1), (u_2,t_2), (u_3,t_3)\Big) \in \left( S^3 \times \rr\right)^4$, we define $S_i = S_{u_i,t_i}$ for $i=0,\dots,3$.

The possible preimages of the $16$ vectors $(S_0,S_1,S_2,S_3)$ define the interiors of regions into which the hyperplanes $H_0,\dots,H_3$ divide $\rr^4$.  Since we work with masses, including the boundary or not does not change the size of the region.

If we are given a mass $\mu$, we can think of $(S_0,\dots,S_3)$ as a random vector, by choosing a point $x\in \rr^4$ with distribution $\mu$ (after normalizing it to a probability measure) and evaluating the vector with $x$.

\begin{definition}
    Given $I \subseteq \{0,1,2,3\}$, a $4$-tuple $\mathcal{H}$ of oriented hyperplanes in $\rr^4$ and a mass $\mu$, we define the Walsh coefficient
    \[
    \mu_{\mathcal{H}}(I) = \int_{\rr^4}\left(\prod_{i \in I}S_i\right)d\mu.
    \]
\end{definition}

If $\mu$ is a probability measure, we have $\mu_{\mathcal{H}}(I) = \mathds{E}\left(\prod_{i \in I}S_i\right)$.  See, e.g., \cite{ODonnell2014} for background on Walsh coefficients and \cite{Simon2015} for applications of Fourier analysis to mass partitions.  The key observation is that $\mathcal{H}$ forms an equipartition of $\mu$ if and only if $\mu_{\mathcal{H}}(I) = 0$ for all non-empty subsets $I \subseteq \{0,1,2,3\}$.  This is because when $I$ is a singleton $\{i\}$, we have $\mu_{\mathcal{H}}(I) = 0$ if and only if $H_i$ bisects $\mu$.  If we have those four equations, and $I = \{i,j\}$, then $\mu_{\mathcal{H}}(I) = 0$ if and only if $H_i$ bisects the restriction of $\mu$ to each side of $H_j$, and so on.

The reason to look at the Walsh coefficients is that it reduces the problem of finding an equipartition to $15$ equations: 
\begin{itemize}
    \item $4$ equations with $I$ being a singleton
    \item $6$ equations with $I$ being a pair
    \item $4$ equations with $I$ being a triple
    \item $1$ equation with $I = \{0,1,2,3\}$.
\end{itemize}

To solve the first ten equations with $\mu = \gamma$ (a normal distribution centered at $0$), the first four equations translate to the hyperplanes going through the origin, and the next six equations translate to the normal vectors being pairwise orthogonal.  For a sufficiently small $\varepsilon$, our choice of $\mu$ given by 
\[
d{\mu} (x) = \Big(1+\varepsilon \cdot \eta(\|x\|^2) \big(P_A(x)+ P_B(x)\big)\Big)d{\gamma}(x),
\]
will imply that the four hyperplanes are close to $0$ and their normal vectors are close to being orthogonal.  We will formalize this statement later.

Let us first look at how the last five equations relate to $A$ and $B$.  Suppose for the moment that $t_0 = \dots = t_3 = 0$ and $(u_0,\dots, u_3) \in O(4)$.  Consider $I = \{1,2,3\}$, and let us compute $\mu_{\mathcal{H}}(I)$.  This is precisely
\begin{align*}
\mu_{\mathcal{H}}(I)  = \int_{\rr^4} S_1S_2S_3 d\mu  & = \int_{\rr^4}  S_1S_2S_3 \Big(1+\varepsilon \cdot \eta(\|x\|^2) \big(P_A(x)+ P_B(x)\big)\Big)d{\gamma}(x) \\
& = \varepsilon \int_{\rr^4}S_1S_2S_3\eta(\|x\|^2) \big(P_A(x)+ P_B(x)\big)d{\gamma}(x)
\end{align*}

The last equality follows from the fact that $\int_{\rr^4}S_1S_2S_3 d{\gamma} = 0$.  Now consider $x=(x_0,x_1,x_2,x_3)$ written in the orthonormal basis $(u_0,u_1,u_2,u_3)$.  In this case, $S_i = \sgn (x_i)$.  Since $\eta (\|x\|^2)$ and $d\gamma$ are even in every coordinate, every monomial in $P_A$ and $P_B$ integrates to $0$ unless it is odd in $x_1,x_2,x_3$ and even in $x_0$.  For $P_A$, since it is a homogeneous polynomial of degree $3$, its only monomial that contributes is the one for $x_1x_2x_3$.  For $P_B$, since it is a homogeneous polynomial of degree $4$, every monomial integrates to $0$.  Therefore, 
\begin{align*}
    \varepsilon \int_{\rr^4}S_1S_2S_3\eta(\|x\|^2) \big(P_A(x)+ P_B(x)\big)d{\gamma}(x) = \varepsilon c_{3,\eta} A(u_1,u_2,u_3),
\end{align*}
for some $c_{3,\eta}>0$ that depends only on $\eta$.

Therefore, to cancel the remaining Walsh coefficients we would need a common zero for the five components in \cref{thm:vanishing-tool}.  It remains to show that the approximation argument we have does not introduce additional problems (provided $\varepsilon$ is small enough), and that \cref{thm:vanishing-tool} is true.

\subsection{The perturbation argument}

Now we show how \cref{thm:vanishing-tool} implies \cref{thm:main}.

\begin{proof}[Proof of \cref{thm:main} assuming \cref{thm:vanishing-tool}]
Consider $d\mu$ and the Walsh coefficients as defined in the previous section.  For $\varepsilon>0$, let $X_\varepsilon = X_{\varepsilon}(\mathcal{H}) \in \rr^{10}$ be the vector consisting of the first ten Walsh coefficients, corresponding to singletons and pairs, and $Y_{\varepsilon}=Y_{\varepsilon}(\mathcal{H})\in \rr^5$ be the vector consisting of the remaining five Walsh coefficients, corresponding to triples and the four-tuple.

An equipartition of the measure corresponds to solving simultaneously $X_{\varepsilon}=0$ and $Y_{\varepsilon}=0$ (each in their corresponding dimension).  Recall that $\mathcal{H}$ is determined by the offset scalar $t_i$ for $0\le i \le 3$ and the normal vectors $u_i$ for $0 \le i \le 3$.

Consider the following ten parameters.  First, each $t_i$ for $0 \le i \le 3$, and second, $\rho_{ij} = \langle u_i, u_j \rangle$ for $0 \le i < j \le 3$.  This gives us a vector
\[
z=(t_0,t_1,t_2,t_3,\rho_{01},\rho_{02},\rho_{03},\rho_{12},\rho_{13},\rho_{23}) \in \rr^{10}.
\]
When $z = 0$, we have $(u_0,u_1,u_2,u_3) \in O(4)$.  The Walsh coefficients depend smoothly on these ten parameters near $z=0$, as the hyperplanes change continuously.  We use $Q$ to denote elements of $O(4)$.

We use $(Q,z)$ as local coordinates for configurations near the Gaussian equipartitions (i.e., near $O(4)$).  More precisely, let $G(\rho)$ be the symmetric matrix with diagonal entries $1$ and off-diagonal entries $\rho_{ij}$.  If all $\rho_{ij}$ are sufficiently small, the matrix $G(\rho)$ is positive definite.  We can then denote $U(Q,\rho) = Q G(\rho)^{1/2}$.  Denote its columns by $u_0,u_1,u_2,u_3$.  Then, $\langle u_i, u_j \rangle = \rho_{ij}$.  Let $\mathcal{H}(Q,z)$ be the four hyperplanes $H(u_i,t_i)$ for the choice of normal vectors $u_i$ we just described.  In what follows, we will use $X_{\varepsilon}(Q,z)$ instead of $X_{\varepsilon}(\mathcal{H}(Q,z))$.

We also note that every configuration $\mathcal{H}$ with normal vectors close enough to $O(4)$ is represented as $\mathcal{H}(Q',z)$ for some pair $(Q',z)$.  This is because, given $H(u_i,t_i)$ for $0\le i \le 3$ with $(u_0,u_1,u_2,u_3)$ sufficiently close to some $Q \in O(4)$, we can construct $U$, the $4 \times 4$ matrix with columns $u_0,\dots,u_3$, and the Gram matrix $G = U^TU$.  If $U$ is close enough to $Q$, then $U$ is invertible.  The polar decomposition gives $U = Q' G^{1/2}$, where $Q' = U(U^TU)^{-1/2}$.  Therefore, $U = U(Q',\rho)$.  The offsets $t_0,\dots,t_3$ then give us $\mathcal{H} = \mathcal{H}(Q',z)$.

For the Gaussian measure, the singleton coefficient associated with $H_i$ is $1-2\Phi (t_i)$, where $\Phi$ is the distribution function of a standard normal random variable.  Therefore, the partial derivative of the Walsh coefficient for $I=\{i\}$ at $z=0$ with respect to $t_i$ is $-2\varphi(0)$, where $\varphi$ is the standard normal density.

For a pair of hyperplanes $H_i, H_j$ going through $0$ (i.e., with $t_i = t_j = 0$), if we denote by $\alpha_{ij}$ the angle between $u_i$ and $u_j$, we have that the Walsh coefficient $\gamma_{\mathcal{H}}(I)$ for $I=\{i,j\}$ is simply $1-\frac{2\alpha_{i,j}}{\pi}$.  This happens because we can project the density to the $2$-dimensional plane spanned by $u_i, u_j$, and in any circle centered at the origin the proportion of points for which $S_i = S_j$ to points for which $S_i \neq S_j$ is $[\pi-\alpha_{i,j} : \alpha_{i,j}]$.  Then, $S_i$ and $S_j$ disagree on two opposite sectors of angular mass $2\alpha_{ij}$, so we have $\mathds{P}[S_i \neq S_j]=\frac{\alpha_{ij}}{\pi}$, and the value of the Walsh coefficient follows.  We also have $\rho_{ij} = \cos (\alpha_{ij})$.  Since $\cos'(\pi/2) = -\sin(\pi/2) = -1$, the derivative of the corresponding Walsh coefficient with respect to $\rho_{ij}$ is $\frac{2}{\pi}$.

At the centered orthogonal Gaussian configuration, the derivatives of a pair coefficient with respect to either offset $t_i, t_j$ vanish, while changing any other pairwise inner product has no effect on that pair coefficient to first order.

This gives us an explicit description of the Jacobian $D_z X_0(Q,0)$.  Namely,
\[
D_z X_0(Q,0) = \begin{bmatrix}
    -2\varphi(0) I_4 & 0 \\
    0 & \frac{2}{\pi}I_6
\end{bmatrix}.
\]
In particular, this matrix is invertible.  The implicit function theorem now says that the first ten equations (i.e., setting the first ten Walsh coefficients to zero) continue to have a unique nearby solution after the measure is perturbed.  Formally, for all sufficiently small $\varepsilon$ and every $Q \in O(4)$, there is a unique small vector $z_{\varepsilon}(Q) \in \rr^{10}$ such that $X_{\varepsilon}(Q,z_{\varepsilon}(Q)) = 0$.  Moreover, $\sup_{Q}\|z_{\varepsilon}(Q)\| \le C|\varepsilon|$ for some constant $C$ independent of $Q$ and $\varepsilon$, and the map $(Q,\varepsilon) \mapsto z_{\varepsilon}(Q)$ is smooth.

This means that the zeros of the first ten Walsh coefficients continue to have zeros that can be parametrized by $O(4)$, even if we no longer have exactly a set of orthogonal normal vectors nor the requirement to go through the origin.

Let us look at the remaining five Walsh coefficients.  Let $R_{\varepsilon}(Q)=Y_{\varepsilon}(\mathcal{H}(Q,z_{\varepsilon}(Q)))$.  An equipartition would require $R_{\varepsilon}(Q)=0$.

However, if we look at the behavior of $Y_0(Q,z)$ near $z=0$, we now have $Y_0(Q,0)=0$ (since orthonormal frames through $0$ form an equipartition of the centered Gaussian) and additionally the Jacobian satisfies $D_z Y_0(Q,0) = 0$.

This follows directly from the independence of the four centered Gaussian coordinates.  If we move one hyperplane, then after differentiating a triple or fourfold sign correlation, the remaining product contains an independent sign with expectation zero.  Similarly, changing one $\rho_{ij}$ affects only two Gaussian coordinates, so a triple or fourfold product still contains an independent centered sign, and the derivative is zero.

With the information about $Y_0(Q,0)$ and its derivatives with respect to directions in $z$, it follows that the change in value from $Y_{\varepsilon}(Q,0)$ and $Y_{\varepsilon}(Q,z_{\varepsilon}(Q))$ is $O(\varepsilon^2)$.  This is the key property that allows us to finish the proof.  Using the fact that $Y_{\varepsilon}(Q,0)$ is linear in $\varepsilon$, uniform Taylor expansion says that
\[
Y_{\varepsilon}(Q,z_{\varepsilon}(Q)) = Y_0(Q,0) + D_zY_0(Q,0)z_{\varepsilon}(Q) + Y_{\varepsilon}(Q,0) +O(\|z_{\varepsilon}(Q)\|^2+|\varepsilon|\|z_{\varepsilon}(Q)\|+ \varepsilon^2).
\]
Put concretely, if we denote $Q=(q_0,q_1,q_2,q_3)$, we have
\[
R_{\varepsilon} (Q) = \varepsilon \begin{pmatrix}
c_{3,\eta}A(q_1,q_2,q_3)\\
c_{3,\eta}A(q_0,q_2,q_3)\\
c_{3,\eta}A(q_0,q_1,q_3)\\
c_{3,\eta}A(q_0,q_1,q_2)\\
c_{4,\eta}B(q_0,q_1,q_2,q_3)
\end{pmatrix} + O(\varepsilon^2),
\]
where $c_{3,\eta}>0$ is the constant described in the previous section, and $c_{4,\eta}>0$ is the corresponding constant from the same analysis using $P_B$ instead of $P_A$.

If the vector $\begin{pmatrix}
c_{3,\eta}A(q_1,q_2,q_3)\\
c_{3,\eta}A(q_0,q_2,q_3)\\
c_{3,\eta}A(q_0,q_1,q_3)\\
c_{3,\eta}A(q_0,q_1,q_2)\\
c_{4,\eta}B(q_0,q_1,q_2,q_3)
\end{pmatrix}$ is nowhere zero on $O(4)$, then a simple compactness argument implies that $R_{\varepsilon}(Q)$ also does not vanish provided that $\varepsilon$ is small enough. 

This argument handles equipartitions in a fixed neighborhood of the Gaussian solution set.  To exclude new solutions elsewhere, notice that all bisecting hyperplanes for $\mu_{\varepsilon}$ have uniformly bounded offsets when $\varepsilon$ is small.  The resulting configuration space is compact.  Since $\mu_\varepsilon \to \gamma$ in total variation, the Walsh coefficients converge uniformly to those for the Gaussian distribution.  No new zeros appear outside an arbitrarily small neighborhood of the centered orthogonal frames for sufficiently small $\varepsilon$, concluding the proof.
\end{proof}

\section{Proof of \cref{thm:vanishing-tool}}\label{sec:non-computational}


Let $e_0, e_1, e_2, e_3$ be the standard basis of $\rr^4$.  Consider the two symmetric matrices
\[
M = \begin{bmatrix}
    0 & 1 & 1 & 1 \\
    1 & 1 & 0 & 0 \\
    1 & 0 & 2 & 0 \\
    1 & 0 & 0 & 4
\end{bmatrix} \qquad \mbox{ and } \qquad N = M + e_0e_0^T = \begin{bmatrix}
    1 & 1 & 1 & 1 \\
    1 & 1 & 0 & 0 \\
    1 & 0 & 2 & 0 \\
    1 & 0 & 0 & 4
\end{bmatrix}
\]
Let $A_0, C \in \Sym^3((\rr^4)^*)$ and $B \in \Sym^4((\rr^4)^*)$ be the tensors associated with the homogeneous polynomials
\[
P_{A_0} = x_0 x^T M x, \qquad P_C = 6x_1x_2x_3, \qquad P_B = (x^T N x)^2.
\]
For $\tau \in \rr$, let $A_{\tau} = A_0 + \tau C$.

Equivalently, the two polynomials we will use are
\begin{align*}
    P_{A_{\tau}}(x) & = 2x_0^2x_1 + 2x_0^2 x_2 + 2x_0^2 x_3 + x_0 x_1^2 + 2x_0 x_2^2 + 4x_0x_3^2 + 6\tau x_1x_2x_3, \\
    P_B(x) &= (x_0^2 + 2x_0x_1 + 2x_0x_2 + 2x_0x_3 + x_1^2 + 2x_2^2 + 4x_3^2)^2 
\end{align*}

For $Q = (q_0, q_1, q_2, q_3) \in O(4)$, consider
\begin{align*}
F_{\tau,0}(Q)&=A_\tau(q_1,q_2,q_3),&
F_{\tau,1}(Q)&=A_\tau(q_0,q_2,q_3),\\
F_{\tau,2}(Q)&=A_\tau(q_0,q_1,q_3),&
F_{\tau,3}(Q)&=A_\tau(q_0,q_1,q_2),
\end{align*}

\begin{theorem}\label{thm:no-computation}
Let \(H_{\tau}(Q) = (F_{\tau,0}(Q), F_{\tau,1}(Q), F_{\tau,2}(Q), F_{\tau,3}(Q), B(q_0,q_1,q_2,q_3))\).  For all sufficiently small $\tau >0$, $H_{\tau}$ has no zero on $O(4)$.
\end{theorem}

Equivalently, for all sufficiently small $\tau > 0$, the tensors $A_{\tau}$ and $B$ satisfy the conclusion of \cref{thm:vanishing-tool}.

There is some room to modify the matrix $M$.  In fact, for any three different real numbers $\lambda_1,\lambda_2,\lambda_3$, if we replace $M$ by
\[
\begin{bmatrix}
    0 & 1 & 1 & 1 \\
    1 & \lambda_1 & 0 & 0 \\
    1 & 0 & \lambda_2 & 0 \\
    1 & 0 & 0 & \lambda_3
\end{bmatrix},
\]
the proof below still holds.  We follow the proof with the values $1,2,4$ for simplicity.

The proof is broken down into two steps.  First, we will characterize the zeros of $H_{\tau}$ when $\tau = 0$, and then we will show that the perturbation introduced by making $\tau >0$ removes the zeros without creating new ones.

\subsection{The zeros when $\tau = 0$.}

Let $\mathcal{G} \subset O(4)$ be the group of signed permutation matrices.

\begin{lemma}
    We have $H_0^{-1}(0) = \mathcal{G}$.
\end{lemma}

Before we prove the lemma above, the reader may notice that this might be slightly unexpected.  If $H_0$ was generic enough and $H_0^{-1}(0)$ was not empty, we would expect it to be $1$-dimensional.  The choice of $M$ turns $H_0^{-1}(0)$ into a discrete set of points, which will be useful once we use $\tau>0$.

\begin{proof}
    For $Q=(q_0,q_1,q_2,q_3) \in O(4)$, let
    \[
    c_i = \langle e_0, q_i\rangle, \qquad m_{ij}=\langle Mq_i, q_j \rangle, \qquad n_{ij}=\langle Nq_i, q_j\rangle.
    \]
    The value $c_i$ is just the first coordinate of $q_i$.  By the definition of $N$, we have $n_{ij}=m_{ij}+c_ic_j$.

    These dot products are important, as we can write $P_{A_0}(x) = \langle e_0, x \rangle \langle Mx, x\rangle$ and $P_B(x) = \langle Nx, x \rangle^2$.

    The first important property of $M$ is that it has no eigenvectors orthogonal to $e_0$.  To show this, suppose for a contradiction that $v=(0,v_1,v_2,v_3)$ and $Mv = \lambda v$.  Then
    \[
    v_1 + v_2 + v_3 = 0, \qquad (\lambda-1)v_1 = (\lambda-2)v_2 = (\lambda-4)v_3 = 0.
    \]
    The equations on the right side mean that at most one of $v_1,v_2,v_3$ can be nonzero.  However, the fact that the sum $v_1 + v_2 + v_3$ is zero implies that $v_1=v_2=v_3=0$, which is the contradiction we wanted.
    
    A direct computation shows that for different $i,j,k$, we have
    \[
    3A_0(q_i, q_j, q_k) = c_i m_{jk} + c_j m_{ik} + c_k m_{ij}.
    \]
    Likewise, for $B$ we have
    \[
    3 B(q_0, q_1, q_2, q_3) = n_{01}n_{23}+ n_{02}n_{13}+ n_{03}n_{12}.
    \]
    We split the proof by how many values among $c_0,c_1,c_2,c_3$ are non-zero.  First, assume that $H_0(Q) = 0$.  The fact that the rows of an orthogonal matrix are unit vectors means that at least one $c_i$ is nonzero.

    \textbf{Case 1.}  If all $c_i$ are nonzero, let $y_{ij} = \frac{m_{ij}}{c_ic_j}$.  If $A_0(q_i,q_j,q_k)=0$, after dividing by $c_ic_jc_k$ we get $y_{ij}+ y_{jk}+ y_{ki}=0$.

    If this happens for every $3$-element subset $\{i,j,k\}$ of $\{0,1,2,3\}$ (i.e., the first four coordinates of $H_0(Q)$ are zero), then we can solve these four linear equations and obtain that there exist $\alpha, \beta, \gamma$ such that $\alpha + \beta + \gamma =0$ and
    \begin{align*}
        y_{01}=y_{23}&=\alpha \\
        y_{02}=y_{13}&=\beta \\
        y_{03}=y_{12}&=\gamma 
    \end{align*}

    When we consider $B$, we now have
    \begin{align*}
    3B(q_0,q_1,q_2,q_3) & = c_0c_1c_2c_3 ((1+\beta)^2 + (1+\alpha)^2 + (1+\gamma)^2) \\
   & = c_0c_1c_2c_3 (\alpha^2 + \beta^2 + \gamma^2 + 3) \neq 0.
   \end{align*}

   \textbf{Case 2.} Suppose that exactly three values of $c_0,c_1,c_2,c_3$ are nonzero.  Without loss of generality, $c_3=0$ and the rest are nonzero.  We now have three equations from $A_0$ that are reduced, giving us
   \begin{align*}
       c_0m_{13}+ c_1m_{03} &=0 \\
       c_0m_{23}+ c_2m_{03} &=0 \\
       c_1m_{23}+c_2m_{13}&=0.
   \end{align*}
   Solving for $m_{13}$ and $m_{23}$ in the first two equations and substituting in the third implies that $m_{03}=0$.  This, in turn, implies that $m_{13}=m_{23}=0$.  In other words, $Mq_3$ is orthogonal to $q_0, q_1, q_2$, which means that $q_3$ is an eigenvector of $M$.  Our condition $c_3=0$ means that $q_3$ is orthogonal to $e_0$, a contradiction.

   \textbf{Case 3.} Now suppose that exactly two values of $c_0,c_1,c_2,c_3$ are nonzero.  We may assume without loss of generality that $c_0,c_1$ are nonzero and $c_2=c_3=0$.  The equation from $A_0$ translate to
   \[
   m_{23}=0, \qquad c_0m_{12}+c_1m_{02}=0, \qquad c_0m_{13}+c_1m_{03}=0.
   \]
   Substituting into the equation for $B$, we have
   \[
   3B(q_0,q_1,q_2,q_3)=m_{02}m_{13}+m_{03}m_{12}=-2\frac{c_1}{c_0}m_{02}m_{03}.
   \]
   If $B(q_0,q_1,q_2,q_3)=0$, at least one of $m_{02},m_{03}$ is zero.  Assume without loss of generality that it is $m_{02}$.  Then $m_{02}=m_{12}=m_{23}=0$, so $q_2$ is an eigenvector of $M$ orthogonal to $e_0$ (again, due to $c_2=0$), which is a contradiction.

   \textbf{Case 4.}  Now suppose exactly one of $c_0, c_1, c_2, c_3$ is nonzero.  Without loss of generality, suppose $c_0 \neq 0$ and $c_1=c_2=c_3= 0$ (the other cases are analogous after relabeling).  Since the rows of an orthogonal matrix are unit vectors, we have $c_0 = \pm 1$ and so $q_0 = \pm e_0$.  In particular, the compression of $M$ to $e_0^{\perp}$ is a diagonal matrix $M'$ with diagonal entries $1,2,4$.  The equations from $A_0$ now become $m_{12}=m_{23}=m_{13}=0$.  Therefore, $q_1,q_2,q_3$ diagonalize $M'$.

   The matrix $M'$ is a diagonal matrix with values $1,2,4$ in the main diagonal.  This implies that $q_1,q_2,q_3$ must be signed permutations of $e_1,e_2,e_3$, as we wanted to show.

   If $Q$ is a signed permutation, then one of its elements is $\pm e_0$, which means that one of the terms $c_i$ will be $\pm 1$ while the rest will be zero.  The terms $m_{jk}$ for distinct $j,k \neq i$ will also be zero.  This causes all the equations from $A_0$ as well as the equation from $B$ to be zero.  Alternatively, notice that $A_0$ and $B$ have no multilinear terms, which is equivalent to the fact that signed permutation matrices evaluate to zero in $H_0$.

   \end{proof}

   For every value of $\tau$, the condition $H_{\tau}(Q)=0$ is invariant under right multiplication by elements of $\mathcal{G}$.  This follows from the fact that $A_{\tau}$ and $B$ are symmetric, so permuting the columns of $Q$ and multiplying them by non-zero constants has no effect on whether $H_{\tau}$ evaluates to zero.

   \subsection{The new perturbation argument.}

   In this subsection, we show that the term $\tau C$ now removes all the zeros from $H$.  Due to the invariance of zeros by $\mathcal{G}$, it is sufficient to understand the effect of the perturbation near the identity.

   This allows us to use a parametrization of the orthogonal matrices close to $I$.  These are standard tools from Lie algebra \cite{Hall2003}, which significantly simplify the analysis below.

   First, the tangent space of $SO(4)$ at $I$ is the set of skew-symmetric matrices:
   \[
   T_ISO(4) = \{X \in M_{4\times4}: X^T = -X\},
   \]
   which follows from differentiating $QQ^T=I$ at $Q=I$.  Given a skew-symmetric matrix $X$, we have $Q=\exp X$ is orthogonal, since $QQ^T = \exp(X)\exp(X^T) = \exp(X)\exp(-X) = I$.  This makes parameterizing the matrices in $SO(4)$ close to $I$ as exponential matrices of a skew-symmetric matrix $X$.  Explicitly, let
   \[
   X(u) = \begin{bmatrix}
       0 & u_{01} & u_{02} & u_{03} \\
       -u_{01} & 0 & u_{12} & u_{13} \\
       -u_{02} & -u_{12} & 0 & u_{23} \\
       -u_{03} & -u_{13} & -u_{23} & 0
   \end{bmatrix} \qquad \mbox{and } Q(u) = \exp (X(u)).
   \]

Every $6$-dimensional vector $u$ corresponds to an orthogonal matrix, and we parametrize every orthogonal matrix sufficiently close to $I$ this way.  Now we analyze the first order Taylor expansions at $(u,\tau)=(0,0)$ of the components of $H_{\tau}(Q(u))$.  Recall that $Q(u) = I + X(u) + O(\|u\|^2)$.  The notation simplifies significantly if we denote $u_{01}=x$, $u_{02}=y$, and $u_{03}=z$.  
\[
I +  X(u) = \begin{bmatrix}
       1 & x & y & z \\
       -x & 1 & u_{12} & u_{13} \\
       -y & -u_{12} & 1 & u_{23} \\
       -z & -u_{13} & -u_{23} & 1
   \end{bmatrix}
\]

This gives us $c_0 =1 + O(\|u\|^2)$, $c_1 = x+ O(\|u\|^2)$, $c_2 = y + O(\|u\|^2)$, and $c_3 = z + O(\|u\|^2)$.

The strategy of the proof will be to use the three equations $F_{\tau,i}(Q(u))=0$ for $i=1,2,3$ to solve for $u_{12}, u_{23}, u_{13}$ in terms of $x,y,z$.  Then, we will use the equation $B=0$ to solve for $x$ in terms of $y,z$, and finally see that the resulting expansion of $F_{\tau,0}$ cannot be zero provided that $\tau>0$ is small enough.  The delicate part of the proof is that we do these reductions while carrying some bounded error terms, which we have to compute carefully to make sure that the final argument gives the desired result.

The computations of $m_{ij}$ are tedious, but computing the linear dependence on $u$ simplifies the process significantly.
\begin{align*}
    m_{12} &= x + y - u_{12} + O(\|u\|^2) \\
    m_{13} &= x + z - 3u_{13}+ O(\|u\|^2) \\
    m_{23} &= y + z - 2u_{23} + O(\|u\|^2)
\end{align*}

Recall that we will compute the terms $c_k m_{ij}$ for $i,j,k$ different.  Since $c_1,c_2,c_3$ are already linear in $u$, we need to approximate $m_{01}, m_{02}, m_{03}$ up to constant terms, which (with a slight abuse of notation) simplifies to
\[
m_{01} = m_{02} = m_{03} = 1 + O(\|u\|).
\]
Therefore, 
\begin{align*}
    3A_0(q_0,q_2,q_3) & = 2y+2z-2u_{23} + O(\|u\|^2), \qquad \mbox{so} \\
    3F_{\tau,1} & = 2y+2z-2u_{23} + O(\|u\|^2+|\tau|\cdot\|u\|)
\end{align*}
This holds because the term from $C$ does not contribute to $A_{\tau}(q_0,q_2,q_3)$ at $u=0$.

Similar analysis gives us

\begin{align*}
    3F_{\tau,2} & = 2x + 2z - 3u_{13} + O(\|u\|^2+|\tau|\cdot\|u\|) \\
    3F_{\tau,3} & = 2x + 2y - u_{12} + O(\|u\|^2+|\tau|\cdot\|u\|).
\end{align*}

Approximating to first order (and using the fact that $n_{01} = n_{02} = n_{03} = 1$ and $n_{12} = n_{23} = n_{13}=0$ at $u=0$), we have
\begin{align*}
3B & = m_{23} + m_{13} + m_{12} + O(\|u\|^2) \\
& = 2x + 2y + 2z - u_{12}-3u_{13}-2u_{23} + O(\|u\|^2)    
\end{align*}

We treat $F_{\tau,0}$ with one more degree of precision, obtaining

\begin{align*}
3F_{0,0} &= x(y+z-2u_{23}) + y(x+z-3u_{13}) + z(x+y-u_{12}) + O(\|u\|^3) \\
& = 2xy + 2xz+ 2yz -2xu_{23}-3yu_{13}-zu_{12} + O(\|u\|^3).    
\end{align*}

The term $6\tau x_1x_2x_3$ contributes $\tau$ at $Q=I$, which shows that at $Q$ it contributes $\tau + O(|\tau| \cdot \|u\|)$.  Therefore,

\[
3F_{\tau,0}(Q(u)) = 3\tau +  2xy + 2xz+ 2yz -2xu_{23}-3yu_{13}-zu_{12} + O(\|u\|^3 + |\tau| \cdot \|u\|).
\]

If we set the equations $F_{\tau,1}=0$, $F_{\tau,2}=0$, $F_{\tau,3}=0$, the linear part is invertible with respect to $u_{12}, u_{13}, u_{23}$.  Therefore, it can be solved uniquely in terms of $v=(x,y,z)$.  More precisely, the derivative of $(3F_{\tau,1},3F_{\tau,2},3F_{\tau,3})$ with respect to $(u_{23},u_{13},u_{12})$ is $\operatorname{diag}(-2,-3,-1)$, which is invertible.  By the implicit function theorem, these three equations determine $u_{23},u_{13},u_{12}$ uniquely as smooth functions of $(x,y,z,\tau)$ near the origin.  When $v=(x,y,z)=(0,0,0)$, the solution is $u_{12}=u_{13}=u_{23}=0$.  The explicit solution is
\begin{align*}
    u_{23} & = y + z + O(\|v\|^2+|\tau|\cdot \|v\|), \\
    u_{13}  & = \frac{2}{3}(x+z) + O(\|v\|^2+|\tau|\cdot \|v\|), \\
    u_{12} & = 2x + 2y + O(\|v\|^2+|\tau|\cdot \|v\|).
\end{align*}

If we set $3B=0$ and substitute the values of $u_{12}$, $u_{13}$, $u_{23}$, we have
\[
0 = -2x-2y-2z + O(\|v\|^2+|\tau|\cdot \|v\|).
\]
We now solve for $x$, set $w = (y,z)$ and have
\[
x = -y - z + O(\|w\|^2 + |\tau|(\|w\|))
\]
Now, once we substitute the values in the equation for $3F_{\tau,0}(Q(u))$, we have
\begin{align*}
    3F_{\tau,0}(Q(u)) & = 3\tau + 2(y^2 + yz + z^2) + O(\|w\|^3 + |\tau|\cdot \|w\|) \\
& =  3\tau + y^2 + z^2 + (y+z)^2 + O(\|w\|^3 + |\tau|\cdot \|w\|) .
\end{align*}

We can choose a neighborhood of $(\tau,w) = (0,0)$ such that the absolute value of the error term is bounded above by $|\tau| + y^2 + z^2$, which will imply that $3F_{\tau,0}(Q(u)) \ge 2\tau + (y+z)^2 \neq 0$ for $\tau > 0$, finishing the proof.

So far, we have shown that close to $(\tau,u) = 0$, we cannot have a zero for $\tau > 0$.  Now we can prove \cref{thm:no-computation}.  For the sake of contradiction, assume that there exists a sequence $\tau_n \to 0$ of positive real numbers and a sequence $Q_n \in O(4)$ such that $H_{\tau_n}(Q_n) = 0$.  By compactness of $O(4)$, we can assume that $Q_n$ converges to some matrix $Q^*\in O(4)$.  Taking the limit, we have that $H_0(Q^*) = 0$, which means that $Q^*$ is a signed permutation matrix.  Consider now the sequence of matrices $Q_n(Q^*)^T$.  We know that $H_{\tau_n}(Q_n(Q^*)^T)=0$.  However, since  $Q_n(Q^*)^T \to I$, for $n$ sufficiently large the pair $(\tau_n, Q_n(Q^*)^T)$ can be written as $(\tau_n , Q(u_n))$ for some sequence $u_n \to 0$.  For large enough $n$, we have proved that $H_{\tau_n}(Q(u_n))$ cannot have a zero, which is a contradiction.
\section{A second proof of \cref{thm:vanishing-tool}}\label{sec:computational}

In this section we give a second proof of \cref{thm:vanishing-tool}.  The approach of this proof is different than the one presented in \cref{sec:non-computational}.  Instead of identifying the zeros of a pair of tensors and then introducing a perturbation, the goal will be to reduce the problem from studying tensors evaluated on $O(4)$ to polynomials evaluated on $[-1,1]^6$.  To show that the resulting polynomials do not share a zero, we will introduce a subdivision argument that certifies that the nonexistence of a zero on certain subsets of $[-1,1]^6$.  Then, a computer program verifies that the subdivision argument covers the entire domain.  Ultimately, this reduces the problem to checking $38,857$ final cases.

We keep this proof in the manuscript because the subdivision argument is interesting by itself, and might be easier to adapt to other instances of mass partition problems.  The proof presented here was the first proof we found of \cref{thm:vanishing-tool}, before finding the tensors in \cref{sec:non-computational} which are much easier to check.  

We introduce a new pair of tensors $A$ and $B$.  Let $P_A$ and $P_B$ be the homogeneous polynomials associated to $A$ and $B$, respectively.  We have $P_A(x) = A(x,x,x)$ and $P_B(x) = B(x,x,x,x)$.  
Take $A$ and $B$ to be the tensors associated with the following polynomials.

\begin{align*}
P_A(x)={}&-13x_0^3+36x_0^2x_1-9x_0^2x_2-15x_0^2x_3
+15x_0x_1^2-45x_0x_2^2+42x_0x_3^2\\
&-8x_1^3-30x_1^2x_2+18x_1x_2^2-\frac65x_1x_2x_3
+18x_1x_3^2-3x_2^3-30x_2x_3^2-13x_3^3.
\\
P_B(x)={}&-24x_0^2x_1x_3-24x_0^2x_3^2+36x_0x_1x_3^2
-12x_0x_2^2x_3\\
&-12x_1^2x_2x_3-6x_1^2x_3^2-12x_1x_2^2x_3
+12x_1x_2x_3^2-24x_2x_3^3.
\end{align*}

Before showing that these two tensors prove \cref{thm:vanishing-tool}, we discuss the intuition that guided their search.  The presentation above has a single multilinear term for $P_A$ and none for $P_B$.

\subsection{Motivation for finding the tensors}\label{subsec:motivation}

Let us explain the strategy used to find $A$ and $B$.  For the sake of the correctness of \cref{thm:main}, this subsection can be skipped.  However, understanding the general path one can follow to find examples like this is perhaps instructive, as opposed to accepting that the two tensors were found by pure luck.  

For a cubic tensor $A \in \Sym^3((\rr^4)^*)$, consider the set of zeros
\[
Z_A = \left\{(q_0,q_1,q_2,q_3) \in SO(4): \begin{bmatrix}
    A(q_1, q_2, q_3) \\
    A(q_0, q_2, q_3) \\
    A(q_0, q_1, q_3) \\
    A(q_0, q_1, q_2)
\end{bmatrix} = \begin{bmatrix}
    0 \\
    0 \\
    0 \\
    0
\end{bmatrix} \right\}
\]

For a generic tensor $A$, transversality suggests that the dimension of $Z_A$ is $\dim SO(4) - 4 = 2$.  The dimension of the zeros of the quartic is five.  We might expect that $Z_A$ and the zeros of $B$ intersect in a curve of zeros.  A pair $A,B$ that satisfies \cref{thm:vanishing-tool} should exploit some special structure from the geometry of $Z_A$.

This is exactly the issue we run into when we try to apply Avis' methods to dimension $4$.  The degrees of freedom give us some breathing room.

Let $\mathcal{G}^+ \subset SO(4)$ be the group of orientation-preserving signed permutation matrices.  An element $h \in \mathcal{G}^+$ acts on an ordered frame by permuting columns and changing some of their signs.  The set $Z_A$ is invariant under this action, by right multiplication.

More precisely, we can write $h \in \mathcal{G}^+$ as a $5$-tuple $h = (\sigma, s_0,s_1,s_2,s_3)$, where $\sigma$ is a permutation and each $s_i \in \{-1,1\}$ determines whether we flip the sign of the $i$-th column.  Consider $\chi(h) = \prod_{i=0}^3 s_i$.

Since $B$ is symmetric, $B (Qh) = \chi(h)B(Q)$.  Therefore, if a connected component of $Z_A$ contains both $Q$ and $Qh$ for some $h \in \mathcal{G}^+$ with $\chi(h) = -1$, then a direct application of the intermediate value theorem shows that no $B$ will satisfy \cref{thm:vanishing-tool}.

Therefore, the search first focused on finding an $A$ where the connected components of $Z_A$ do not connect $Q$ and $Qh$ when $\chi(h) = -1$.

Now, given $Q = (q_0,q_1,q_2,q_3)$ and $B$, a convenient way to describe $B(q_0,\dots,q_3)$ is as the dot product of tensors $\langle B, v(Q) \rangle$ with 
\[
v(Q) = q_0\odot q_1\odot q_2\odot q_3 = \frac{1}{4!}\sum_{\sigma \in S_4} q_{\sigma(0)}\otimes q_{\sigma(1)}\otimes q_{\sigma(2)}\otimes q_{\sigma(3)}.
\]
Then, if no connected component of $Z_A$ contains elements $Q$ and $Qh$ with $\chi(h) = -1$, we can create a sign assignment on $Z_A$.  More precisely, a function $s:Z_A \to \{-1,1\}$ such that $s$ is constant on each connected component of $Z_A$ and such that $s(Qh) = \chi(h)s(Q)$ for all $h \in \mathcal{G}^+$ and $Q \in Z_A$.  This reduces the existence of a $B$ as we want to showing that $0 \not\in \conv \{s(Q)v(Q) : Q \in Z_A\}$ in $\Sym^4(\rr^4)$.

This changes the problem of finding $A$ and $B$ in a vacuum, to first finding the tensor $A$ as we look for some properties of the connected components of $Z_A$, and then finding $B$ by finding a hyperplane (in $\Sym(\rr^4)$ as a vector space) that separates $0$ from $\{s(Q)v(Q) : Q \in Z_A\}$.

If this could be done formally, it would constitute a full proof of \cref{thm:vanishing-tool}.  Instead, what was done was to test candidates for $A$, sample points from $Z_A$, and then use a separation argument as above with the sampled points to obtain candidates for $B$.  This is not a complete argument for why the resulting tensors work, but it provides a way to obtain the candidates.  Once we have these candidates, we can follow the steps in the next section to prove \cref{thm:vanishing-tool}.

\subsection{Why the tensors work.}

Let us first describe the strategy we follow to prove \cref{thm:vanishing-tool}, assuming we have $A$ and $B$.  We will do the analysis in $SO(4)$ instead of $O(4)$, as it is not a material change for finding zeros of $A$ and $B$ (flipping one column does not affect whether a frame evaluates to zero or not).

First, we will use a parametrization to go from $SO(4)$ to the hypercube $[-1,1]^6$.  In general, we do not have a clean parametrization of $SO(4)$ with $[-1,1]^6$.  However, we can use the fact that $A$ and $B$ are symmetric to reduce our particular problem to $[-1,1]^6$.  The tensors will translate to polynomials of bounded degree with $6$ variables.  The problem will be reduced to showing that these polynomials do not all share a zero.

To show that the polynomials do not share a zero, we will subdivide $[-1,1]^6$ into smaller hypercubes, and show that, restricted to each of those hypercubes the polynomials do not have a zero using Bernstein coefficients.

Bernstein range bounds, de Casteljau subdivision, and subdivision solvers for polynomial systems have a rich history \cites{Garloff1986, Farouki1988, Garloff1993, Sherbrooke1993, Mourrain2009, Farouki2012}.

We identify $\rr^4$ with the quaternions $\hh$.  If $p,r \in \hh \setminus \{0\}$, the map
\[
x \longmapsto \frac{px\overline r}{|p||r|}
\]
is an orientation-preserving orthogonal transformation.  Every element of $SO(4)$ has this form, and
\begin{align*}
    \Spin(4) = S^3 \times S^3 & \longrightarrow SO(4) \\
    (p,r) & \longmapsto (x \mapsto px\overline{r}) 
\end{align*}
is a double cover (see, e.g., \cite{Coxeter1946}).  If we denote by $Q(p,r)$ the transformation in $SO(4)$ generated by $p,r$, and $e_0 = 1, e_1 = i, e_2 = j, e_3 = k$, the columns of $Q(p,r)$ are $\tilde{q}_i = pe_i \overline{r}/(|p||r|)$.  Multiplication of $p,r$ by a quaternionic unit sign-permutes its four coordinates and the columns of $Q(p,r)$.  Therefore, using this operation we can assume that the first coordinate of $p$ and $r$ is the one with largest absolute value.

We introduce one final new notation to build $\hat{q_i} = pe_i\overline{r}$ and define
\begin{align*}
    f_0(p,r) &= A(\hat{q}_1,\hat{q}_2,\hat{q}_3) \\
    f_1(p,r)&= A(\hat{q}_0,\hat{q}_2,\hat{q}_3) \\
    f_2(p,r)&=A(\hat{q}_0,\hat{q}_1,\hat{q}_3)\\
    f_3(p,r)&=A(\hat{q}_0,\hat{q}_1,\hat{q}_2)\\
    f_4(p,r)&=B(\hat{q}_0,\hat{q}_1,\hat{q}_2,\hat{q}_3).
\end{align*}

Note that the maps above simultaneously vanish if and only if the equivalent maps with $\tilde{q}_i$ instead of $\hat{q}_i$ simultaneously vanish.  We can use this fact to choose a different normalization of the double-cover of $\Spin(4)$.  

The equations are bihomogeneous in $p,r$ (i.e., the pair $(p,r)$ gives the same element of $SO(4)$ (up to sign) as $(\lambda p, \lambda' r)$ for nonzero scalar $\lambda, \lambda'$).  Therefore, we can effectively parametrize $SO(4)$ with $\rr \mathds{P}^3 \times \rr \mathds{P}^3$, the product of two projective spaces.  Then, given $[p_0:p_1:p_2:p_3] \in \rr \mathds{P}^3$, we can choose a canonical representative by dividing by the coordinate with largest absolute value.  This would usually require several charts covering $\rr \mathds{P}^3$.  However, for the purpose of studying the zero sets $Z_A$ and $Z_B$, since the tensors are symmetric, we can move the coordinate with the largest absolute value to the first position as described above.  This allows us to parametrize
\[
[p_0:p_1:p_2:p_3] \longleftrightarrow (1,x_1,x_2,x_3) \quad |x_i| \le 1
\]
In other words, every possible common zero of our map in $SO(4)$ is represented by a point in the hypercube $[-1,1]^6$.  We denote the elements of $[-1,1]^6$ as $(x_1,x_2,x_3,y_1,y_2,y_3)$.  With this new parametrization, the maps $f_0,\dots, f_4$ described above are polynomials on the variables $x_i,y_j$.  The first four, induced by $A$, become polynomials of degree at most $3$ in the $x$ variables and at most $3$ in the $y$ variables.  The application of $B$ becomes a polynomial of degree at most $4$ in the $x$ coordinates and at most $4$ in the $y$ coordinates.

This induces a polynomial map
\begin{align*}
    f:[-1,1]^6 &\to \rr^5. \\
    f&=(f_0,\dots,f_4)
\end{align*}

\cref{thm:vanishing-tool} is equivalent to showing that $f$ has no zeros.  The first four polynomials are induced by $A$, the last polynomial is induced by $B$.  To show that $f$ has no zeros, we use Bernstein coordinates.

To motivate this approach, let us explain the process in one variable.  We will do this for polynomials of degree at most four, since those are the ones we will work with.  For $0 \le k \le 4$, let
\[
\beta_k(t) = \binom{4}{k}t^k(1-t)^{4-k}, \qquad t \in [0,1].
\]
There are two properties of these functions that we use: first, $\beta_k(t) \ge 0$ for all $t \in [0,1]$, and second, $\sum_{k=0}^4 \beta_k(t) = 1$.

Every polynomial $h$ of degree at most $4$ can be written uniquely as
\[
h(t) = \sum_{k=0}^4 b_k \beta_k(t).
\]
Therefore, $\min_k b_k \le h(t)$ for all $t \in [0,1]$.  If all $b_k > 0$, we can guarantee that $h(t)$ does not have a zero in the interval $[0,1]$.  For higher dimension, we first reparametrize $[-1,1]^6$ as $[0,1]^6$ using the linear change of variable $z_j = \frac{1}{2}+\frac{1}{2}w_j$ for each of the six coordinates of $[-1,1]^6$.  Let $\tilde{f}_i$ for $i=0,\dots,4$ be the corresponding polynomials after this change of variables and $\tilde{f} = (\tilde{f}_0,\dots,\tilde{f}_4)$.

For a multi-index $\alpha = (\alpha_1,\dots,\alpha_6) \in \{0,\dots,4\}^6$ we define
\[
\beta_{\alpha}(z) = \prod_{j=1}^6 \beta_{\alpha_j}(z_j).
\]
These form a basis of polynomials of degree at most $4$ in each variable.  In other words, for $s=0,\dots,4$, we can write each $\tilde{f}_s$ as
\[
\tilde{f}_s(z) = \sum_{\alpha \in \{0,\dots,4\}^6}b_{s,\alpha}\beta_{\alpha}(z),
\]
where $b_{s,\alpha}$ are scalars.  Now consider the vectors $C_{\alpha} = (b_{0,\alpha},\dots,b_{4,\alpha}) \in \rr^5$ for all $\alpha \in \{0,\dots,4\}^6$, which we call the \textit{Bernstein vectors} of $\tilde{f}$.

\begin{lemma}\label{lem:separating-bernstein}
    For every $z \in [0,1]^6$, we have $\tilde{f}(z) \in \conv \{C_{\alpha}: \alpha \in \{0,\dots,4\}^6\}$.
\end{lemma}

\begin{proof}
    This follows directly from the fact that, for $z \in [0,1]^6$, the values $\beta_{\alpha}(z)$ for $\alpha \in \{0,\dots,4\}^6$ are the coefficients of a convex combination.
\end{proof}

Therefore, to show that $f$ does not have a zero, it would be sufficient to prove that $0 \not\in \conv \{C_{\alpha}: \alpha \in \{0,\dots,4\}^6\}$.  By the standard convex separation theorem, this is equivalent to finding a vector $u \in \rr^5$ such that $\langle u, C_{\alpha}\rangle > 0$ for every $\alpha \in \{0,\dots,4\}^6$.

Finding such a vector $u$ is too ambitious.  However, an alternative approach is to subdivide $[0,1]^6$ into smaller axis-parallel boxes, reparametrize each of them as $[0,1]^6$, and apply the same idea.

The computer-assisted part of this manuscript gives a subdivision of $[0,1]^6$ into smaller axis-parallel boxes and a separation vector $u$ for each box showing that the five polynomials obtained from $A$ and $B$ do not have a common zero in that box.  This is described in more detail in the next section.

\begin{proof}[Proof of \cref{thm:vanishing-tool}]
    Provided a subdivision of $[0,1]^6$ with separating vectors as stated above, \cref{lem:separating-bernstein} shows that the five polynomials induced by $A$ and $B$ do not have a common zero.  This, in turn, corresponds to the map from \cref{thm:vanishing-tool} having no zero on $SO(4)$, which in turn implies it has no zeros on $O(4)$.
\end{proof}

\section{Computer-assisted component}\label{sec:final-computer-part}

Let us describe the algorithm used to find the subdivision of $[0,1]^6$.

The algorithm consists of the following:

\begin{enumerate}
    \item Start with the full hypercube $[0,1]^6$, set as the ``current box''.
    \item Compute the Bernstein vectors of $\tilde{f}$ in the current box.
    \item Try to find a vector $v \in \rr^5$ that has positive dot product with each Bernstein vector for $\tilde{f}$.
    \item If such a vector $v$ is found, mark the current box as ``completed''.
    \item If no vector $v$ is found, divide the current box in half along one coordinate and mark the two new boxes as ``remaining''.
    \item Set a remaining box as the current box and start the process again from step (2).
\end{enumerate}

Subdividing a box by half in one coordinate allows us to use the de Casteljau subdivision formulas to simplify the computation of the Bernstein vectors of the smaller boxes.  If the algorithm terminates, we have a certificate that $\tilde{f}$ has no simultaneous zeros.  The algorithm also imposes a tree structure on the subdivision of $[0,1]^6$, where each subdivision creates two nodes, one for each new box, which are the children of the current box.

A Python program was used to run the algorithm described above.  It created a tree with $77,713$ nodes and $38,857$ leaves.  See \cref{fig:subdivision} for an example of a slice of the subdivision of $[-1,1]^6$.

\begin{figure}
    \centering
    \includegraphics[width=0.5\linewidth]{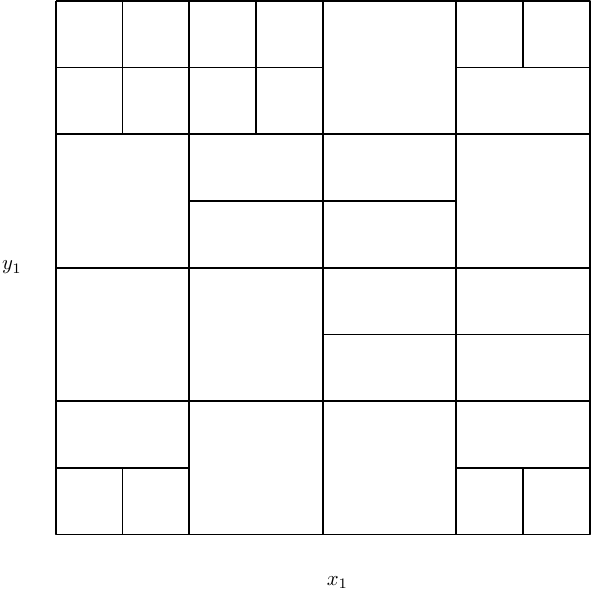}
    \caption{A 2-dimensional slice of the subdivision of $[-1,1]^6$.  This slice is obtained by fixing $x_2=-1/3$, $x_3=1/3$, $y_2 = -1/3$, $y_3=-1/3$.  The horizontal axis is $x_1$ and the vertical axis is $y_1$.}
    \label{fig:subdivision}
\end{figure}

We include in the accompanying repository:

\begin{itemize}
    \item A C program that gives the final subdivision of $[0,1]^6$, the vector $v$ needed for each final box.  The program reconstructs $f$, the Bernstein vectors for each box, and checks the dot product condition.
    \item The Python files that were used to find the subdivision, and additional information for reproducibility.
\end{itemize}

The supporting files are as follows:

\begin{center}
\small
\begin{tabular}{>{\raggedright\arraybackslash\scriptsize}p{0.49\textwidth}>{\raggedright\arraybackslash}p{0.41\textwidth}}
\toprule
File&Role\\
\midrule
\path{verify_grunbaum_tensors.c}&Self-contained exact checker for \cref{thm:vanishing-tool}; it reconstructs the polynomial system and Bernstein vectors from the tensors $B$ and $5A$ given in \cref{sec:computational} (multiplication by $5$ is only done to have integer coefficients) and it verifies the complete subdivision certificate using integer arithmetic.\\
\path{verification_manifest.md}&Hashes, compiler information, expected output, and inventory of the additional reproducibility files\\
\path{grunbaum_leaf_boxes.csv}&A CSV file with the coordinates of every final box and the vector used to separate the Bernstein vectors from the origin.  This file contains a subdivision of the hypercube $[0,1]^6$, after the first change of variables.\\
\bottomrule
\end{tabular}
\end{center}

The CSV file makes independent inspection, plotting, and reimplementation easier.  It lists all 38,857 terminal boxes explicitly.  For each coordinate, the entries $(k,d)$ represent the interval $\left[ \frac{k}{2^d}, \frac{k+1}{2^d}  \right]$.  A row describing a box will use the format

\texttt{leaf\_id,node\_id,depth,x1\_k,x1\_d,...,y3\_k,y3\_d,w0,w1,w2,w3,w4}

where \texttt{x1\_k, x1\_d} represent the interval in the $x_1$ coordinate as described above, and the vector $(w_0,\dots,w_4)$ used to separate the Bernstein vectors from the origin.

The repository also contains the Python programs used to find the subdivision and separators. These files are provided for reproducibility and independent comparison, but they are not needed to verify the certificate: the C checker is self-contained, and the CSV file provides the data without any need to trust the algorithm that produced the subdivision.

The program \texttt{verify\_grunbaum\_tensors.c} contains a Base64 encoding of the subdivision tree and leaf witnesses.  A second C file, \texttt{check\_grunbaum\_text.c}, performs the same test but reads the boxes and vectors information from \texttt{grunbaum\_certificate.txt}.

\vspace{0.5cm}

\begin{center}
\textbf{Repository: }\url{https://github.com/psoberon-math/grunbaum-tensor-verification}
\end{center}

\section{Remarks}

We note one of the main differences between Avis' earlier counterexamples and those presented in this manuscript.  The counterexamples by Avis focus on global information about the masses.  In other words, if we concentrate the mass around the moment curve, it is the intersection properties of the moment curve with hyperplanes that imply that the construction is indeed a counterexample.

In our construction, the example is constructed from a very small perturbation on a centered Gaussian mass.  Every equipartition of the centered Gaussian mass corresponds to an orthogonal family that contains the origin, and a small (but carefully chosen) perturbation near the origin makes a fix impossible.

The relation of equipartitions of small perturbations of the standard Gaussian mass with common zeros from some symmetric tensors works in general.  In particular, cases where Ramos' conjecture is known to hold might give interesting consequences about the existence of common zeros of polynomials.  We have not found cases with interesting applications, but there might be some instances worth exploring.  The construction seems tailor-made for the case $d=4$.  As a trivial example, a homogeneous polynomial of degree $5$ in $n \ge 5$ variables has multiple possible monomials which are odd in three of the variables but even in the rest.  This affects significantly the steps to characterize equipartitions with zeros from certain evaluations of a tensor.

The structure of the analytic proof is reminiscent of standard equivariant proofs of Borsuk–-Ulam-type theorems \cites{Barany1980, Musin2012}, where one first constructs a test map with a single orbit of zeros and then proves that this orbit cannot disappear under equivariant perturbation. Here \(H_0\) likewise has exactly one orbit of zeros under the signed-permutation action. The difference is that a well chosen perturbation can now remove the entire zero orbit. The same kind of highly symmetric map that is usually used to prove existence is used here to demonstrate nonexistence.

\subsection*{Disclosure of use of AI}

ChatGPT-5.6 Sol was used during the discovery of this result.  The initial counterexample strategy and the tensors used in \cref{sec:computational} were found during conversations with ChatGPT, the final form of the tensors used in \cref{sec:non-computational} were also improved with ChatGPT.  The proof was substantially changed and simplified by the author.  The author reviewed and executed the final verification software and wrote the manuscript.  The author has verified every mathematical claim and computational output used in the final proof and assumes full responsibility for the results presented.

\bibliography{refref}

\end{document}